\documentclass[11pt]{amsart}
\usepackage[margin=1.0in]{geometry}
\usepackage{setspace}

\usepackage[utf8]{inputenc}
\usepackage{amsmath,amsfonts,amssymb,amsopn,amscd,amsthm,bbm}
\usepackage{comment}
\usepackage{dsfont}
\usepackage{graphicx}
\usepackage{color}
\usepackage[colorlinks]{hyperref}
\usepackage{epigraph,todonotes}
\usepackage{enumerate}
\allowdisplaybreaks[4]

\DeclareMathOperator{\Hc}{ \textrm{ch} }

\def\1{{\mathbf 1}}

\def\pa{\partial}

\def\a{\alpha}
\def\b{\beta}

\def\d{\delta}
\def\e{\epsilon}

\def\N{{\mathbb N}}
\def\Z{{\mathbb Z}}

\def\R{{\mathbb R}}

\def\P{{\mathbb P}}
\def\E{{\mathbb E}}

\def\Hc{{\mathcal H}}

\def\Lc{{\mathcal L}}

\def\Pc{{\mathcal P}}

\def\Vc{{\mathcal V}}

\numberwithin{equation}{section}

\newtheorem{thm}{Theorem}[section]
\newtheorem{proposition}{Proposition}[section]

\newtheorem{definition}{Definition}[section]
\newtheorem{example}{Example}[section]
\newtheorem{lemma}{Lemma}[section]

\newtheorem{rmk}{Remark}[section]

\newcommand{\ignore}[1]{}
\newcommand{\vertiii}[1]{{\left\vert\kern-0.25ex\left\vert\kern-0.25ex\left\vert #1 
    \right\vert\kern-0.25ex\right\vert\kern-0.25ex\right\vert}}

\newcommand\smallO{
  \mathchoice
    {{\scriptstyle\mathcal{O}}}% \displaystyle
    {{\scriptstyle\mathcal{O}}}% \textstyle
    {{\scriptscriptstyle\mathcal{O}}}% \scriptstyle
    {\scalebox{.7}{$\scriptscriptstyle\mathcal{O}$}}%\scriptscriptstyle
  }

\title{A PDE approach for regret bounds under partial monitoring} 
\author{Erhan Bayraktar}\thanks{E. Bayraktar is partially supported by the National Science Foundation under grant DMS-2106556 and by
the Susan M. Smith chair.}
\address{Department of Mathematics, University of Michigan}
\email{erhan@umich.edu}
\author{Ibrahim Ekren}\thanks{I. Ekren is supported in part by NSF Grant DMS 2007826.}
\address{Department of Mathematics, Florida State University}
\email{iekren@fsu.edu}
\author{Xin Zhang} 
\address{Department of Mathematics, University of Vienna}
\email{xin.zhang@univie.ac.at}

\begin{document}

\maketitle

\begin{abstract}
In this paper, we study a learning problem in which a forecaster only observes
partial information. By properly rescaling the problem, we heuristically derive a limiting PDE on Wasserstein space which characterizes the asymptotic behavior of the regret of the forecaster. Using a verification type argument, we show that the problem of obtaining regret bounds and efficient algorithms can be tackled by finding appropriate smooth sub/supersolutions of this parabolic PDE.
\end{abstract}

\section{Introduction}

In this paper, we study a zero-sum game between a forecaster and an adversary. At each round, the forecaster chooses an action between $K\geq 2$ alternative actions based on his partial observations aiming at performing as well as the best constant strategy, while the adversary aims at maximizing the forecaster's regret. Our problem is motivated by \emph{prediction with expert advice}  and  \emph{bandit problems} (see e.g. \cite{cesa2006prediction,BubeckCesaBianchi2012}), which are fundamental problems in online learning and sequential decision making. The main difference between prediction with expert advice and bandit problem is the information observed by the forecaster. In prediction with expert advice problems, the forecaster can monitor the outcomes of each alternative action, whereas in bandit problems, the forecaster can only observe the outcome of the action chosen. Thus, the former problem is a full information game whereas the latter is a bandit game (see e.g. \cite{audibert2010regret,audibert2011minimax}).

The most commonly used algorithm for decision making and prediction problem is the so-called \emph{multiplicative weights algorithm}, which assigns initial weights to each expert, update these weights multiplicatively and iteratively based on their performance, and randomly choose experts according to their weights. This simple algorithm is widely used and has been proven efficient in practice. However, it cannot provide accurate regret bounds and best strategies for the forecaster. In \cite{MR4053484}, techniques from partial differential equations were first employed to understand asymptotic behavior of prediction of expert advice problems. Since then, it became popular and has been proven powerful in certain problems, see e.g. \cite{https://doi.org/10.1002/cpa.22071,2022arXiv220205767K,pmlr-v125-kobzar20a,10.5555/3546258.3546330,MR4187120,MR4120922,9203984,9317932,2022arXiv220107877Z,pmlr-v167-greenstreet22a}.

In full information games, these papers rely on the fact that the difference $(X^i_t)_{i=1,\ldots, N}=(G^i_t-G_t)_{i=1,\ldots, N}\in \R^K$ between the gain $G_t$ of the forecaster and the gain $G^i_t$ of each action $i$ is a natural state variable for the dynamic game between the forecaster and the adversary. Thus, the minimax regret of the forecaster satisfies a finite dimensional dynamic programming principle whose scaling limit is a parabolic partial differential equation on $\R^N$. For bandit games or in the presence of partial information such methodology cannot be applied. Indeed, due to partial information, the natural state variable for the dynamic programming principle is the set of probability distributions on $\R^N$ which encodes the distribution $m_t$ of $X_t$ conditional on the information of the forecaster. Thus, with partial information, the fundamental problem is to understand the dynamics of $m_t$ and how these dynamics behave in the long-time regime. 

Our main contribution consists in showing that the update of the conditional distribution between two consecutive time steps from $m_t$ to $m_{t+1}$ admits a scaling limit that can be described using partial differential equations in the Wasserstein space. The equations we obtain are fully nonlinear versions of the PDEs appearing in mean-field games and Mckean-Vlasov control problems, see e.g. \cite{cdll2019,cosso2021master,MR3907014}. This novel relation between the discrete-time bandit problem and the continuous-time equations comes from the fact that in the game we study, the updated measure $m_{t+1}$ can be written as a push-forward operator on $m_t$, i.e. $m_{t+1}=(Id+Y_t)\sharp m_t$ where $Y_t$ is a (random) function describing the feature learned by the forecaster on $[t,t+1]$. If the game is played $T$ times and if we rescale the problem with its natural $\sqrt{T}$ scaling, the update of the $m_t$ can be written as $m_{t+\frac{1}{T}}=(Id+\frac{Y_t}{\sqrt{T}})\sharp m_t$. In the long-time regime, i.e., as $T\to\infty$, by the definition of the Wasserstein derivative (see \cite[Proposition 2.3]{cdll2019}), we obtain that for any smooth function $U$, we have the expansion $$U\left(m_{t+\frac{1}{T}}\right)=U\left(m_t\right)+ \frac{1}{\sqrt{T}}\int D_m U(m_t,x)Y_t(x) \, m(dx)+\smallO \left(\frac{1}{\sqrt{T}}\right).$$

Thus, the impact of the Bayesian update of the distribution $m_t$ can be characterized in the long-time regime using the Wasserstein derivative $D_m U$. In fact, we derive a second order expansion of $U(m_{t+\frac{1}{T}})$ involving the derivatives $D_xD_mU$ and $D_{mm}^2 U$ which allows us to heuristically exhibit a second order parabolic equation of type 
{\small\begin{align}\label{eq:pdeintro}
0&=\pa_t U(t,m)\\
&+F\left(\int D_m U(t,m,x)m(dx),\int D_xD_m U(t,m,x)m(dx),\iint D^2_{mm} U(t,m,x,y)m(dx)m(dy)\right)\notag
\end{align}}

\noindent which is expected to govern the dynamics of the prediction problem in the long-time regime. In this equation, the unknown is the function $U$, $F$ is a function that can be explicitly computed from the Bayes' rule, and the derivatives are defined as in \cite{cdll2019,cosso2021master}.

The equation \eqref{eq:pdeintro} gives simple methods to obtain algorithms and regret bounds for the long-time regime of the prediction problem with partial information. Indeed, using a verification type argument we show that the gradient  $D_m U$ of any smooth supersolution $\phi$ of \eqref{eq:pdeintro} satisfying some growth condition yields to an algorithm that guarantees an upper bound for regret of order $\phi(0,\delta_0)\sqrt{T}$ where $\delta_0$ is the Dirac mass at $0$. A similar result also holds for appropriate subsolutions. 

Due to the nonlinearity on the second derivative term $D_{mm}^2 U$, wellposedness of viscosity solutions for \eqref{eq:pdeintro} is not available in the literature. 
Hence, the questions of establishing appropriate comparison result for viscosity solutions and obtaining the exact growth of the regret as for example in \cite{MR4053484} are left for future research.

The rest of paper is organized as follows. In Section 2, we formulate our problem and show that the value function of the game depends only on the law $m_t$ of $X_t$ conditional on the information of the agents. Then, using Bayes' rule, we compute explicitly the update of beliefs and prove a dynamic programming principle. In Section 3, by properly rescaling the value function and using differential calculus on the space of measures, we heuristically obtain a limiting PDE of type \eqref{eq:pdeintro} on the Wasserstein space. In Section 4 and 5, using smooth supersolutions and subsolutions of the PDE, we construct strategies for the forecaster and the adversary, and find upper and lower bounds of expected regret.

\subsection{Notations}
For any positive integer $K$, define $[K]=\{1,\dotso, K\}$, and $\mathcal{S}_K$ to be the set of positive semidefiniete $K\times K$ matrices.  $Id$ stands for the identity mapping of appropriate dimension. For any $x \in \R^K$, denote its $i$-th coordinate by $x^i$. Let $\{e_i: \, i=1,\dotso, K\}$ be the canonical basis of $\R^K$, and for any $j \subset [K]$, denote $e_j=\sum_{i, i \in j} e_i$ and $e=\sum_{i=1}^Ke_i$.

We fix $K\geq 2$ and denote by $\Pc_2(\R^K)$ the set of probability measures $ m $ on $\R^K$ such that $\int |x|^2  \, m (dx)<\infty$.  For any $v\in \R^K$, $\lambda\in \R$, and $ m \in \Pc_2(\R^K)$, we define the measures $ m _{\sharp v}:=(Id+v)\sharp  m $ and $m ^{*\lambda}$ via
\begin{align*}
\int f(x)  \, m _{\sharp v}(dx)&=\int f(x+v) \, m  (dx),\\
 \int f(x) \, m ^{*\lambda}(dx)&=\int f(\lambda x) \, m  (dx)\mbox{ for all }f\mbox{ continuous and bounded}.
\end{align*}
Additionally, for any function $f$ and $m\in \Pc_2(\R^K)$, we denote
\begin{align*}
f([m]):=\int f(x) \, m(dx) .
\end{align*}

\section{Formulation of the problem}

Our online prediction problem with partial observation can be described as a $T$-round game, played by a forecaster in an adversarial environment. Suppose that there are $K$ actions. At each round $t$, the forecaster chooses an action $I_t \in [K]$, and independently the adversary chooses the reward $J^i_t$ of action $i$ to be $0$ or $1$, i.e., $J_t \in \{0,1\}^K$. Then the total gain of the forecaster $G_t$ and the total gain $G_t^i$ of action $i$ evolve as 
\begin{align*}
&G_{t+1}-G_t= \1_{I_t \in J_t }, \\
&G^i_{t+1}-G^i_t=\1_{i \in J_t}, \, \, \,  i=1,\dots,K. 
\end{align*}
The goal of the forecaster is to design a robust strategy that performs as well as the best constant strategy under any adversarial environment, i.e., to minimize $ \max\limits_{Adversary} \E[ \max_i X_T^i]$, where $X_t^i:=G^i_t-G_t$ is the state variable evolving as 
\begin{align*}
X_{t+1}-X_{t}:=e_{J_{t}}-\1_{I_{t} \in J_{t}}  e  \in \R^K.
\end{align*}
Both the forecaster and the adversary are allowed to adopt randomized strategies. At each round $t$, they decide on  distributions $b_t$ of $I_t$ and $a_t$ of $J_t$ respectively. If we allow both agents to observe the outcomes of $I_t$ and $J_t$, this problem is the classical prediction with expert advice problem in the adversarial setting, see for example \cite{cover1966behavior,cesa2006prediction,peres,MR4053484}.

Let us now describe information observed by the forecaster and his admissible strategies in the partial information problem we aim to study. At initial time $t=0$, both the adversary and the forecaster get informed of the distribution $m_0$ of $X_0$. For any $t \geq 0$, the random variable
\begin{align*}
Y_t:= \1_{I_{t} \in J_{t}} {I_{t}}- \1_{I_{t} \not \in J_{t}} {I_{t}} \in \{\pm i \}
\end{align*} 
indicates whether the forecaster makes a good decision or not. Both players can observe the law of adversary's control $a_{t-1}$ and the indicator $y_{t-1}$. Their accumulated information is given by 
\begin{align*}
h_t:=(m_0, a_0, y_0 \dotso, a_{t-1}, y_{t-1}) \in \mathcal{H}_t, \quad (h_0:=m_0 \in \mathcal{H}_0),
\end{align*}
where $\mathcal{H}_t:=\Pc(\R^K) \times \left( \Pc(\{0,1\}^K) \times \{\pm i \}\right)^{t}.$ 
The strategies of the forecaster and the adversary are measurable functions $\b_t: \mathcal{H}_t \to \Pc([K])$ and $\a_t: \mathcal{H}_t \to \Pc(\{0,1\}^K)$ respectively. Define $\mathcal{A}$ to be the set of all possible strategies $\alpha:=(\alpha_0,\alpha_1,\dotso, \alpha_{T-1})$, and $\mathcal{B}$ similarly.

Suppose this game starts from time $t$ with an initial distribution $m \in \Pc(\R^K)$. Then given any strategies $\alpha \in \mathcal{A}$, $\beta \in \mathcal{B}$, the regret for the forecaster is given by 
\begin{align*}
\gamma_T(t,m, \a, \b):= \E^{m ,\a , \b} [ \max_{i} X^i_T \, | \, X_t \sim m].
\end{align*} 
From the perspective of the forecaster, we aim at solving a minimax problem 
\begin{align}\label{eq:minmax}
v_T(t,m):= \inf_{\beta \in \mathcal{B}} \sup_{\a \in \mathcal{A}} \gamma_T (t,m, \a, \b),
\end{align}
and we denote this two player game by $\Gamma_T(t,m)$. 

\begin{rmk}
(i) This formulation is motivated by the classical bandit problems, see for example \cite{cesa2006prediction,BubeckCesaBianchi2012}. Similar to the bandit problems, both agents have access to the same partial information and they simultaneously choose their strategies to be played at each round. Before each round, our adversary has the same information as the non-oblivious adversary of \cite{audibert2010regret}. However, unlike the bandit problems where the agents only observe $Y_t$, after each round, the forecaster also learns the strategy $a_t$ chosen by the adversary in the previous round. Therefore, both players does not get access to state variable $X_t$, but still can compute the same update of its law conditional to a common filtration. 

(ii) We note that the observation of $a_t$ by the forecaster has a practical implication in terms of design of recommender systems. In the problem that we envision, the forecaster recommends successively one item among $K$ alternatives to a sequence of users. The forecaster's strategy is the choice of a randomization $b_t\in \Pc([K])$ which leads to the realization of a random variable $I_t\in [K]$ representing the recommendation of the forecaster. A sequence of users arrive and decide whether or not they accept the successive recommendations of the forecaster. We identify each user with $a_t\in \Pc(\{0,1\}^K)$ so that the realization of random variable $J_t \in \{0,1\}^K$ represents the random choice that each successive user makes. Thus, unlike stochastic bandit problems, $a_t$ is chosen adversarially at each round and learned by the forecaster after the round. In this context, the assumption that the forecaster learns $a_t$ means that after proposing $I_t$, the forecaster learns the identity of the user $a_t$.

(ii) In our context, since the forecaster learns $a_t$, he can update $m_t$ via Bayes' rule. This update is impossible in the classical bandit problems. An interesting question that is left for future research is to extend our PDE tools to allow such an uncertainty on the update procedure of the conditional distribution.

\end{rmk}

\subsection{Dynamic programming principle}

In this subsection, we establish the dynamic programming principle for the game \eqref{eq:minmax}, and reduce controls $\a,\b$ to functions of conditional distribution of the state $X$. Let us first compute the distribution of $X$, i.e., belief, given prior information. Suppose the current distribution is $m$ and $X$ is a random variable with distribution $m$. We denote $\Delta X$ the change of $X$ between two rounds. The players choose strategies $a \in \Pc(\{0,1\}^K)$ and $b \in \Pc([K])$ respectively, and receive signal $y \in \{ \pm i \}$. We denote by $\Lc^{a,b}$ the distribution of a random variable and by $\P^{a,b}$ the probability of an event given the strategies of the agents. We omit the superscripts $a$ or $b$ if this dependence is clear from the context. We also denote by $$ l (m,a,y):= \mathcal{L}^{a}( X+\Delta X  |  X \sim m,  Y=y) \in \Pc(\R^K)$$ the Bayesian update of the distribution. 

We will compute the explicit formula of $l(m,a,y)$ in the next Lemma.  For any $a \in \Pc(\{0,1\}^K)$, $b \in \Pc([K]) $, denote 
\begin{align*}
\hat a(i):= \sum_{j, i\in j} a(j),\quad \hat a(-i):= \sum_{j, i\notin j} a(j),  \quad \forall \, i \in [K].
\end{align*}

\begin{proposition}\label{prop:updatebelief}
Given $a \in \Pc(\{0,1\}^K)$ and the distribution $m \in \Pc(\R^K)$, we have that  
\begin{align*}
l(m,a,i)=\sum_{j,i \in j} \left(\frac{a(j)}{\hat a(i)}m\right)_{\sharp -e_{j^c}}, \quad l(m,a,-i)=\sum_{j,i \notin j} \left(\frac{a(j)}{\hat a(-i)}m\right)_{\sharp e_j}.
\end{align*}
We make the convention in these expressions that $l(m,a,y)= \delta_{\mathbf{0}} \in \Pc(\R^K)$ whenever $\hat a(y)=0$. 
\end{proposition}
\begin{proof}
For $j\subset [K]$ and $i\in [K]$, it can be easily verified that
\begin{align*}
\P(\Delta X=e_j, Y=-i, X\in dx)&=\1_{i\notin j}a(j)b(i)m(dx) \\
\P(\Delta X=-e_{j^c}, Y=i, X\in dx)&=\1_{i\in j}a(j)b(i)m(dx) \\
\P( Y=i)&=b(i)\sum_{k, i\in k}a(k)\\
\P(Y=-i)&=b(i)\sum_{k,i\notin k}a(k) \\
\P(\Delta X=e_j, X\in dx \, | \, Y=-i)&=\frac{\1_{i\notin j}a(j)m(dx)}{\sum_{k,i\notin k}a(k)}\\
\P(\Delta X=-e_{j^c}, X\in dx \, | \,Y=i)&=\frac{\1_{i\in j}a(j)m(dx)}{\sum_{k,i\in k}a(k)}.
\end{align*}
Therefore, conditioning on $Y$, the distribution of $X+\Delta X$ is given by 
\begin{align*}
 \P \left[(X + \Delta X) \in dx \, | \,  Y =i \right] 
 &= \sum_{j, i \in j} \P \left[X \in d(x+e_{j^c}), \Delta X= -e_{j^c}  \, | \,  Y=i  \right]  \notag \\
&=\sum_{j,i \in j} \frac{a(j)m(d(x+e_{j^c}))}{\hat a(i)},\notag\\
& =\sum_{j,i \in j} \left(\frac{a(j)}{\hat a(i)}m\right)_{\sharp -e_{j^c}} (dx),
\end{align*}
and 
\begin{align*}
\P \left[(X + \Delta X) \in dx \, | \, Y =-i \right] 
&= \sum_{j,i \notin j} \P\left[X\in d(x-e_{j}), \Delta X= e_{j}  \, | \,  Y=-i  \right] \notag \\
&=\sum_{j,i \notin j} \frac{a(j)m(d(x-e_{j}))}{\hat a(-i)}\notag\\
&=\sum_{j,i \notin j} \left(\frac{a(j)}{\hat a(-i)}m\right)_{\sharp e_j} (dx).
\end{align*}
\end{proof}

The following theorem proves a dynamic programming principle showing that one can solve \eqref{eq:minmax} with a backward induction. 

\begin{thm}\label{thm:dpp}
For any distribution $m \in \Pc(\R^K)$ and $T \in \N $ we have that 
\begin{align}\label{eq:dpp}
v_T(t,m)=& \inf_{ b \in \Pc([K])} \sup_{a \in \Pc(\{0,1\}^K)} \left(\sum_{i=1}^K   b(i)\hat a(i) v_{T}(t+1, l(m,a,i)) \right. \notag \\
&+ \left. \sum_{i=1}^K  b(i) \hat a(-i) v_{T}(t+1, l(m,a,-i)) \right),
\end{align}
where $b(i) \hat a(i)$, $b(i) \hat a(-i)$ represent the probability of receiving signal $i$, $ -i$ respectively, and $l(m,a,\pm i)$ is the update of beliefs. 
\end{thm}
\begin{proof}
The equation \eqref{eq:dpp} holds trivially for $t=T-1$. Suppose it is true for $t+1$. Let us prove it for $t$. Denote 
by $v$ the value of the right hand side of \eqref{eq:dpp}. 
For any $\a \in \mathcal{A}$ and $ \b \in \mathcal{B}$, denote $\a_{t+1:T}=\{\a_{t+1}, \dotso, \a_{T-1}\}$, $\b_{t+1:T}= \{ \b_{t+1}, \dotso, \b_{T-1}\}$. 
It is clear that 
\begin{align}\label{eq:induction}
\gamma_T(t,m, \a, \b)=& \sum_{i=1}^K  \sum_{k: i \in k} \b_t(i)\a_t(k) \gamma_{T}(t+1, l(m,\a_t,i), \a_{t+1:T},\b_{t+1:T}) \notag \\
&+ \sum_{i=1}^K  \sum_{k: i \notin k} \b_t(i)\a_t(k) \gamma_{T}(t+1, l(m,\a_t,-i), \a_{t+1:T},\b_{t+1:T}),
\end{align}
where $l(m, \a_t, \pm i)$ is the conditional distribution of $X_{t+1}$. For the game 
$\gamma_{T}(t+1, l(m,\a_t, \pm i))$, due to our induction hypothesis, the value of this game exists and is just $v_{T}(t+1,l(m,\a_t, \pm i))$. Taking supremum over $\a$ on both sides of \eqref{eq:induction}, it can be easily seen that 
\begin{align*}
\sup_{\a} \gamma_T(t,m, \a, \b) \geq & \sup_{\a_t} \left( \sum_{i=1}^K  \sum_{k: i \in k} \b_t(i)\a_t(k) v_{T}(t+1,l(m,\a_t,i))\right. \\
&+\left. \sum_{i=1}^K  \sum_{k: i \notin k} \b_t(i)\a_t(k) v_{T}(t+1,l(m,\a_t,-i)) \right).
\end{align*}
Taking infimum over $\b$, we conclude that $v_T(c,m) \geq v$.

Then we prove that for any $\e >0$, there exists a robust strategy $\b^*$ of the forecaster such that 
\begin{align}\label{eq:upp}
\sup_{\a} \gamma_T(t,m,\a,\b^*) < v+2\e. 
\end{align}
Take $\b_t^* \in \Pc([K])$ with the property that 
\begin{align*}
v+\e> & \sup_{a \in \Pc(\{0,1\}^K)} \left(\sum_{i=1}^K  \sum_{k: i \in k} \b_t^*(i)a(k) v_{T}(t+1, l(m,a,i)) \right. \notag \\
&+ \left. \sum_{i=1}^K  \sum_{k: i \notin k} \b_t^*(i)a(k) v_{T}(t+1,l(m,a,-i)) \right).
\end{align*}
By induction hypothesis, for any belief $l(m,a, \pm i)$, the forecaster can choose a strategy $\b_{t+1:T}^*$ such that 
\begin{align*}
v_{T}(t+1,l(m,a, \pm i))+\e > \sup_{\a_{t+1:T}} \gamma_{T}(t+1,l(m,a, \pm i), \a_{t+1:T}, \b_{t+1:T}^*).
\end{align*}
Taking $\b^*=(\b_t^*, \b^*_{t+1:T})$, clearly it is measurable and  satisfies \eqref{eq:upp}. 
\end{proof}

\section{Heuristic expansion of the rescaled value function}

Let us define the rescaled value functions  
\begin{align*}
u^T(s,m):=\frac{1}{\sqrt{T}} v_T\left({\lceil sT \rceil}, m^{*\sqrt{T}} \right),
\end{align*}
and equivalently 
\begin{align*}
v_T({\lceil sT \rceil},m )=\sqrt{T}u^T\left(s, m^{*\sqrt{T^{-1}}} \right).
\end{align*}
For any $a \in \Pc(\{0,1\}^K)$ and belief $ m  \in \Pc(\Z^K)$, denote 
\begin{align}\label{eq:defA}
A^{a, m }_{i,\sqrt{T}}&= \left({\sum_{j, i \in j} \left(\frac{a(j)}{\hat a(i)}  m ^{*\sqrt{T}}\right)_{\sharp -e_{j^c}} } \right)^{*\frac{1}{\sqrt{T}}}, \,  \, \, \, \,  
A^{a, m }_{-i,\sqrt{T}}&=\left(  \sum_{j, i \notin j} \left(\frac{a(j)}{\hat a(-i)}  m ^{*\sqrt{T}}\right)_{\sharp e_{j}} \right)^{*\frac{1}{\sqrt{T}}}.\end{align}
Then due to \eqref{eq:dpp}, it holds that 
\begin{align}\label{eq:dppT}
u^T\left( s-\frac{1}{T},m\right) =& \inf_{b \in \Pc([K])}  \sup_{a\in \Pc(\{0,1\}^K)}  \bigg( \sum_i b(i) \hat a(i) u^T\left(s,A^{a, m }_{i,\sqrt{T}} \right) \notag \\
&+\sum_i b(i) \hat a(-i) u^T\left( s,A^{a, m }_{-i,\sqrt{T}} \right) \bigg),
\end{align}
with the terminal condition 
\begin{align*}
u^T(1, m) = \int_{x \in \R^K} \max_{i} {x^i}  \, m(dx). 
\end{align*}

Now we want to derive a limit for \eqref{eq:dppT} as $T \to \infty$. This derivation  requires us to take derivatives in the direction of $A^{a,m}_{i,\sqrt T}-m$ and $A^{a,m}_{-i,\sqrt T}-m$ in the Wasserstein space. Let us introduce the differentiability of functions over the Wasserstein space as defined in \cite{cdll2019,carmona2018probabilistic}.

A function $u : \Pc_2(\R^K)\mapsto \R$ is said to be Fr\'{e}chet differentiable if 
there exists a continuous function
$$\frac{\d u}{\d m}:\Pc_2(\R^K)\times \R^K\mapsto \R$$ so that for all $(m,m') \in \Pc_2(\R^K)$, we have that
\begin{align*}
\lim_{h\to 0}\frac{u(m+h (m'-m))-u(m)}{h}=\int \frac{\d u}{\d m}(m,x) \, (m'-m)(dx).
\end{align*}
Whenever $\frac{\delta u}{\delta m}$ is differentiable in $x$, we also define 
\begin{align*}
D_m u(m,x)=D_x\frac{\d u}{\d m}(m,x)\in \R^K.
\end{align*}
Similarly to \cite{cdll2019,chow2019partial}, we can also define $D_xD_m u(m,x)$, $D^2_{mm}u(m,x,y)$.

\begin{definition}
A function $u: \Pc_2(\R^K) \to \R$ is said to be $\mathcal{C}^1$ if $D_m u(m ,x)$ is continuous and has at most quadratic growth in $x$, i.e., 
\begin{align*}
| D_m u(m,x)| \leq C(1+|x|^2). 
\end{align*}
It is said to be $\mathcal{C}^2$ if $D_xD_m u (m,x)$ and $D^2_{mm} u(m,x,y)$ are continuous, and have at most quadratic growth in $x$ and $(x,y)$ respectively. 
\end{definition}

It is shown in \cite[Proposition 2.3]{cdll2019} that $D_m u$ can be understood as a derivative of $u$ along push-forward directions, meaning that for all Borel measurable bounded vector field $\phi:\R^K\mapsto \R^K$
we have
$$\lim_{h\to 0}\frac{u((Id+h \phi)_\sharp m)-u(m)}{h}=\int D_mu(m,x)\phi(x) \, m(dx).$$
Due to the expression of  $A^{a,m}_{i,\sqrt T}$ and $A^{a,m}_{-i,\sqrt T}$, we need to take derivatives in the directions $\left(Id+\frac{e_{j}}{T}\right)$ which are constant vector fields. However, the presence of terms $\frac{a(j)}{\hat a(i)}\sharp m$ in \eqref{eq:defA} is a randomization among the directions of the vector fields. The following Proposition shows that at the leading order, we can simplify these perturbations by averaging over these different vector fields. We recall the notational convention that for all $m'\in \Pc_2(\R^K)$
\begin{align*}
D_m u(m,[m'])&=\int D_m u(m,x) \, m'(dx) \in \R^K.
\end{align*}

\begin{proposition}\label{p:firstorder}
Suppose $u \in \mathcal{C}^{1}(\Pc(\R^K); \R)$. Then for all $a \in \Pc(\{0,1\}^K)$ and $i\in [K]$, we have that
\begin{align*}
 \lim_{T\to \infty} \sqrt T\left(u(A^{a,m}_{i,\sqrt T})-u(m)\right) &=-\Vc_{a,i}^\top D_mu\left(m,[m]\right) \\
 \lim_{T\to \infty}\sqrt T\left(u(A^{a,m}_{-i,\sqrt T})-u(m)\right) & =\Vc_{a,-i}^\top D_mu\left(m,[m]\right),
\end{align*}
where 
\begin{align*}
\Vc_{a,i}: =\sum_{j:i\in j}\frac{a(j)}{\hat a(i)} e_{j^c}\in \R^K, \quad \Vc_{a,-i}: =\sum_{j:i\notin j}\frac{a(j)}{\hat a(-i)} e_{j}\in \R^K.
\end{align*}
\end{proposition}
\begin{rmk}
Note that $-\Vc_{a,i}\in \R^K$ (resp. $\Vc_{a,-i}\in \R^K$ ) represents the increase in the expectation of $X_t$ given the information that $Y=i$ (resp. $Y=-i$) and the adversary's strategy $a$.
\end{rmk}

\begin{proof}
Let us only compute the derivative in the direction of $A^{a,m}_{i,\sqrt T} -m$. By the definition of $\frac{\d u}{\d m}$, denoting $\widetilde A_{s,\sqrt T,m}=m+s (A^{a,m}_{i,\sqrt T}-m)$ we have that
\begin{align*}
&\sqrt T(u(A^{a,m}_{i,\sqrt T})-u(m))\\
&=\sqrt T\int_0^1 \int \frac{\d u}{\d m}\left(\widetilde A_{s,\sqrt T,m},x\right) \, (A^{a,m}_{i,\sqrt T}-m)(dx) \, ds\\
&=\sum_{j:i\in j}\frac{a(j)}{\hat a(i)} \sqrt T\int_0^1 \int \frac{\d u}{\d m}\left(\widetilde A_{s,\sqrt T,m},x-\frac{e_{j^c}}{\sqrt T}\right)-\frac{\d u}{\d m}\left(\widetilde A_{s,\sqrt T,m},x\right) \, m(dx) \,ds,
\end{align*}
and thus
\begin{align*}
\lim_{T\to \infty}\sqrt T\left(u(A^{a,m}_{i, \sqrt T})-u(m)\right)=-\sum_{j:i\in j}\frac{a(j)}{\hat a(i)} \int e_{j^c}^\top D_mu\left(m,x\right) \, m(dx).
\end{align*}
\end{proof}

We can now give the second order expansion along $T \mapsto u(A^{a,m}_{y, \sqrt T})$ for all $ y= \pm i$. 

\begin{proposition}\label{p:secondorder}
Suppose $u \in C^{2}(\Pc(\R^K); \R)$. Then we have that 
\begin{align}
&\lim_{T\to \infty}T\left(u(A^{a,m}_{i,\sqrt T})-u(m)+\frac{1}{\sqrt T}\Vc_{a,i}^\top D_mu\left(m,[m]\right)\right)\notag\\
&=\frac{1}{2}\sum_{j:i\in j}\frac{a(j)}{\hat a(i)}  e_{j^c}^\top D_{x}D_mu\left(m,[m]\right)e_{j^c}\label{eq:exp1}\\
& \, \, \, +{\frac{1}{2}}  \sum_{k,j:i\in k,i\in j}\frac{a(j)}{\hat a(i)}\frac{a(k)}{\hat a(i)} {e_{j^c}^\top} D^2_{mm}u\left(m,[m],[m]\right)e_{k^c},\notag
\end{align}
and 
\begin{align}
&\lim_{T\to \infty}T\left(u(A^{a,m}_{-i,\sqrt T})-u(m)-\frac{1}{\sqrt T}\Vc_{a,-i}^\top D_mu\left(m,[m]\right)\right)\notag\\
&=\frac{1}{2}\sum_{j:i\notin j}\frac{a(j)}{\hat a(-i)}  e_{j}^\top D_{x}D_mu\left(m,[m]\right)e_{j}\label{eq:exp2}\\
&+{\frac{1}{2}} \sum_{k,j:i\notin k,i\notin j}\frac{a(j)}{\hat a(-i)}\frac{a(k)}{\hat a(-i)} {e_{j}^\top} D^2_{mm}u\left(m,[m],[m]\right)e_{k}\notag
\end{align}
\end{proposition}
\begin{rmk}
The Propositions \ref{p:firstorder} and \ref{p:secondorder} show that, at the leading orders, the impact of the scaled update $A^{a,m}_{y,\sqrt T}$ of $m$ on a smooth function $u$ can be characterized by multiplication of $D_m u$, $D_xD_mu$, and $D_{mm}^2 u$ with some matrices depending only on $a$.
\end{rmk}
\begin{proof}
Using the \cite[Equality (25)]{cdll2019}, we have
\begin{align*}
&T \left(u(A^{a,m}_{i,\sqrt T})-u(m)+\frac{1}{\sqrt T}\sum_{j:i\in j}\frac{a(j)}{\hat a(i)} \int e_{j^c}^\top D_mu\left(m,x\right)dm(x)\right)\\
&=\sum_{j:i\in j}\frac{a(j)}{\hat a(i)}T\int_0^1 \int \frac{\d u}{\d m}\left(\widetilde A_{s,\sqrt T,m},x-\frac{e_{j^c}}{\sqrt T}\right)-\frac{\d u}{\d m}\left(\widetilde A_{s,\sqrt T,m},x\right) \\
&\, \, \, \, \, \, \, +\frac{e_{j^c}^\top}{\sqrt T} D_x\frac{\d u}{\d m}u\left(m,x\right) dm(x)  ds.
\end{align*}
Let us compute the limit of integrand on the right hand side.  By Taylor expansion on $x$, it can ben seen that 
\begin{align*}
 & T\left(\frac{\d u}{\d m}\left(\widetilde A_{s,\sqrt T,m},x-\frac{e_{j^c}}{\sqrt T}\right)-\frac{\d u}{\d m}\left(\widetilde A_{s,\sqrt T,m},x\right)+\frac{e_{j^c}^\top}{\sqrt T} D_x\frac{\d u}{\d m}u\left(m,x\right)\right)\\
 &=\frac{1}{2}e_{j^c}^\top D^2_{x}\frac{\d u}{\d m}\left(\widetilde A_{s,\sqrt T,m},\widetilde x_{T}\right)e_{j^c}-\sqrt T{e_{j^c}^\top} \left(D_x \frac{\d u}{\d m}\left(\widetilde A_{s,T,m},x\right)-D_x\frac{\d u}{\d m}u\left(m,x\right)\right)
 \end{align*}
 where $\widetilde x_T$ is some point on the line segment joining $x$ and $x-\frac{e_{j^c}}{\sqrt T}$. Denoting $\widetilde A_{r,s,\sqrt T,m}=m+r (\widetilde A_{s,\sqrt T,m}-m)$, the right hand side of the above equation equals to
 \begin{align*}
 & \frac{1}{2}e_{j^c}^\top D^2_{x}\frac{\d u}{\d m}\left(\widetilde A_{s,\sqrt T,m},\widetilde x_{T}\right)e_{j^c}- s\sum_{k:i\in k}\frac{a(k)}{\hat a(i)} \int_0^1 \int  \\
 &\sqrt T  \bigg({e_{j^c}^\top}D_x\frac{\d^2 u}{\d m^2}\left(\widetilde A_{r,s,\sqrt T,m},x,x'-\frac{e_{k^c}}{\sqrt T}\right) 
 -{e_{j^c}^\top}D_x\frac{\d^2 u}{\d m^2}\left(\widetilde A_{r,s,\sqrt T,m},x,x'\right)\bigg) \, dm(x') dx.
\end{align*}
Letting $T \to \infty$, it converges to
\begin{align*}
 \frac{1}{2}e_{j^c}^\top D^2_{x}\frac{\d u}{\d m}\left(m,x\right)e_{j^c}+s \sum_{k:i\in k}\frac{a(k)}{\hat a(i)} {e_{j^c}^\top}\int D^2_{x,x'}\frac{\d^2 u}{\d m^2}u\left(m,x,x'\right)e_{k^c} \, dm(x'),
\end{align*}
and hence we obtain \eqref{eq:exp1} by integrating over $x$. Similar computation yields to \eqref{eq:exp2}.
\end{proof}
We now use \eqref{eq:dppT} to obtain a formal asymptotics for $u^T$ as $T\to\infty$. Assuming $u^T$ converges to a $\mathcal{C}^2$ function $u: [0,1] \times \Pc(\R^K) \to \R$, the dynamic programming principle yields to  
\begin{align*}
0&= \inf_{b \in \Pc([K])} \sup_{a\in \Pc(\{0,1\}^K)}\sum_i b(i) \hat a(i)T\left( u^T\left(s,A^{a,m}_{i,\sqrt{T}}\right)-  u^T\left( s-\frac{1}{T},m\right)\right)\\
&\quad+b(i)\hat a(-i)T\left(  u^T\left( s,A^{a,m}_{-i,\sqrt{T}} \right)-  u^T\left( s-\frac{1}{T},m\right)\right).
\end{align*}
Using Proposition~\ref{p:firstorder} and \ref{p:secondorder} for large enough $T$, we obtain that 
\begin{align}
&\mathcal{O}(1)=\pa_t  u(t,m)+\inf_{b \in \Pc([K])} \sup_{a\in \Pc(\{0,1\}^K)} \sqrt{T}\sum_i b(i) \left(\hat a(-i)\Vc_{a,-i}-\hat a(i)\Vc_{a,i} \right)^\top D_m u\left(t,m,[m]\right) \notag 
\\
&+\frac{1}{2}{b(i)}{\hat a(i)}\left(\Vc_{a,i}^\top D^2_{mm} u\left(t,m,[m],[m]\right)\Vc_{a,i}+\sum_{j:i\in j}\frac{a(j)}{\hat a(i)}e_{j^c}^\top D_{x}D_m u\left(t,m,[m]\right)e_{j^c}\right)\notag\\
&+\frac{1}{2}{b(i)}{\hat a(-i)} \left(\Vc_{a,-i}^\top D^2_{mm} u\left(t,m,[m],[m]\right)\Vc_{a,-i}+\sum_{j:i\notin j}\frac{a(j)}{\hat a(-i)}  e_{j}^\top D_{x}D_m u\left(t,m,[m]\right)e_{j}\right).\label{eq:wpdeup} 
\end{align}

Notice that $A^{a,m_{ \sharp \epsilon \1}}_{y, \sqrt T}= \left(A^{a,m}_{y, \sqrt T}\right)_{\sharp \epsilon \1}$ for any $ y \in \{ \pm i\}$, and the final condition satisfies $u^T(1,m_{\sharp \epsilon \1})=u^T(1,m)+\epsilon$. Therefore by backward induction, we have $u^T(t,m_{\sharp \epsilon \1})=u^T(t,m)+\epsilon$ for any $t \in [0,1]$, and also in its limit as $T \to \infty$ 
\begin{align*}
u(t,m_{\sharp \epsilon\1})=u(t,m)+\epsilon.
\end{align*}
Thus, thanks to \cite[Proposition 2.3]{cdll2019}, we have that 
\begin{align*}
\1^\top D_mu(t,m,[m])= 1.
\end{align*}
Additionally, each component of $D_mu(t,m,[m])$ is clearly non-negative, which implies that $D_mu(t,m,[m])\in \R^K$ is simplex valued. 
Denoting $u_i(t,m)$ the $i$th component of $ D_mu\left(t,m,[m]\right)$, we have that 
\begin{align}\label{eq:minimax}
&\sum_i b(i)  \left(\sum_{j:i\notin j}{a(j)} e_{j}^\top-\sum_{j:i\in j}{a(j)}e_{j^c}^\top \right) D_mu\left(t,m,[m]\right) \notag \\
&=\sum_i b(i)  \left(\sum_{j}{a(j)} e_{j}^\top-\sum_{j:i\in j}{a(j)}\1^\top  \right) D_mu\left(t,m,[m]\right) \notag \\
&= \sum_{j}{a(j)} \sum_{i \in j}u_i(t,m)-\sum_i b(i) \sum_{j:i\in j}{a(j)} = \sum_{i}(u_i(t,m)-b(i))\sum_{j:i \in j}{a(j)}.
\end{align}
Thus, in order to have the equality \eqref{eq:wpdeup}, the coefficients of the $\sqrt T$ term must be zero, i.e., 
\begin{align*}
0&=\inf_{b \in \Pc([K])} \sup_{a\in \Pc(\{0,1\}^K)} \sum_i b(i) \left(\hat a(-i)\Vc_{a,-i}-\hat a(i)\Vc_{a,i} \right)^\top D_m u\left(t,m,[m]\right) \\
&=\inf_{b \in \Pc([K])} \sup_{a\in \Pc(\{0,1\}^K)} \sum_{i}(u_i(t,m)-b(i))\sum_{j:i \in j}{a(j)}. 
\end{align*}
Otherwise, the first order term explodes. Therefore the forecaster is forced to choose the strategy $b= D_mu\left(t,m,[m]\right)$, and we obtain the PDE 
\begin{align}\label{eq:originalPDE}
&0=\pa_t  u(t,m) +\sup_{a\in \Pc(\{0,1\}^K)}  \sum_i \\
&+\frac{1}{2}{u_i(t,m)}{\hat a(i)}\left(\Vc_{a,i}^\top D^2_{mm} u\left(t,m,[m],[m]\right)\Vc_{a,i}+\sum_{j:i\in j}\frac{a(j)}{\hat a(i)}e_{j^c}^\top D_{x}D_m u\left(t,m,[m]\right)e_{j^c}\right)\notag\\
&+\frac{1}{2}{u_i(t,m)}{\hat a(-i)} \left(\Vc_{a,-i}^\top D^2_{mm} u\left(t,m,[m],[m]\right)\Vc_{a,-i}+\sum_{j:i\notin j}\frac{a(j)}{\hat a(-i)}  e_{j}^\top D_{x}D_m u\left(t,m,[m]\right)e_{j}\right). \notag
\end{align}

\begin{rmk}\label{rmk:balanced}
(i) We say $a \in \Pc(\{0,1\}^K)$ is a balanced strategy if $ \sum\limits_{j: i \in j} a(j)$ is independent of $i$, and denote by $\mathcal{E}$ the set of all balanced strategies. According to \eqref{eq:minimax}, if we restrict $a$ in \eqref{eq:wpdeup} to be balanced, the first order term vanishes for any $b \in \Pc([K])$. 

(ii)The standard tool to show the convergence of $u^T$ to the solution of \eqref{eq:originalPDE} is to use the stability and comparison of viscosity solutions, see for example \cite{https://doi.org/10.1002/cpa.22071,barles1991convergence} in the finite dimensional cases.  However, a comparison result for viscosity solution of second order PDEs on the Wasserstein space is not available in the literature in the generality we need, see for example \cite{burzoni2020viscosity,MR3907014,cosso2021master} and the references therein. 

(iii) Because the second derivative terms $D^2_{mm} u$ and $D_{x}D_m u$ are expected to explode as $t\to1$, the generator of  \eqref{eq:originalPDE} is expected to become discontinuous as $t\to 1$. Thus, it is more convenient to use the equation
\begin{align}\label{eq:originalPDEu}
&0=\pa_t  u(t,m) +\sup_{i,a\in \Pc(\{0,1\}^K)}   \\
&+\frac{1}{2}{\hat a(i)}\left(\Vc_{a,i}^\top D^2_{mm} u\left(t,m,[m],[m]\right)\Vc_{a,i}+\sum_{j:i\in j}\frac{a(j)}{\hat a(i)}e_{j^c}^\top D_{x}D_m u\left(t,m,[m]\right)e_{j^c}\right)\notag\\
&+\frac{1}{2}{\hat a(-i)} \left(\Vc_{a,-i}^\top D^2_{mm} u\left(t,m,[m],[m]\right)\Vc_{a,-i}+\sum_{j:i\notin j}\frac{a(j)}{\hat a(-i)}  e_{j}^\top D_{x}D_m u\left(t,m,[m]\right)e_{j}\right)\notag
\end{align}
to obtain regret bounds. Indeed, any supersolution of \eqref{eq:originalPDEu} is clearly a supersolution of \eqref{eq:originalPDE} and the generator of \eqref{eq:originalPDEu} is Lipschitz continuous on the derivatives of $u$. Thus, one can expect a simpler proof of comparison of viscosity solutions. 

\end{rmk}

\section{Upper bound by smooth supersolution of the PDE}
In this part, we design robust strategies of the forecaster using smooth supersolutions of \eqref{eq:originalPDE}. Note that \eqref{eq:originalPDE} becomes simpler if $D^2_{mm} u=0$. This is the case if $u$ is linear in $m$. The following Lemma uses this idea to generate simple supersolutions to \eqref{eq:originalPDE}.
\begin{lemma}\label{lem:simplesuper}
Let $ \phi$ be a classical solution of
\begin{equation}\label{eq:simplepde}
\begin{split}
&0 \geq \pa_t   \phi(t,x)+\frac{1}{2}\sup_{i, a\in \Pc(\{0,1\}^K)} Tr\left(D^2_{xx}\phi\left(t,x\right)\left(\sum_{j}{a(j)}\left( \1_{i\in j}e_{j^c}e_{j^c}^\top+\1_{i\notin j}e_{j}e_{j}^\top \right)\right)\right)  \\
 &\phi(1,x) \geq \max_i  x^i, \, \, \, \phi(t,x + \lambda \1)=\phi(t,x)+\lambda.
 \end{split}
\end{equation}
 Then, the function $\Phi:[0,1]\times \Pc_2(\R^K)\mapsto \R$ defined by 
$$ \Phi(t,m)=\phi(t,[m]):=\int \phi(t,x) \, m(dx)$$
is a smooth supersolution to \eqref{eq:originalPDE}
with 
\begin{equation}\label{eq:derreduced}
D_{m} \Phi\left(t,m,x\right)=D_x   \phi(t,x),\, \, D_xD_{m} \Phi\left(t,m,x\right)=D^2_{xx}   \phi(t,x),\, \,   D^2_{mm} \Phi\left(t,m,x,y\right)=0.
\end{equation}
\end{lemma}

\begin{proof}
Using \eqref{eq:derreduced} which can be easily verified, together with the supersolution property of $\phi$ we have that 
\begin{align*}
0& \geq \pa_t  \phi(t,[m])+\frac{1}{2}\int \sup_{i, a\in \Pc(\{0,1\}^K)} Tr\left(D^2_{xx} \phi\left(t,x\right)\left(\sum_{j}{a(j)}\left( \1_{i\in j}e_{j^c}e_{j^c}^\top+\1_{i\notin j}e_{j}e_{j}^\top \right)\right)\right)dm(x)\\
&\geq \pa_t  \phi(t,[m])+\frac{1}{2} \sup_{i, a\in \Pc(\{0,1\}^K)} Tr\left(D^2_{xx}\phi\left(t,[m]\right)\left(\sum_{j}{a(j)}\left( \1_{i\in j}e_{j^c}e_{j^c}^\top+\1_{i\notin j}e_{j}e_{j}^\top \right)\right)\right)\\
&\geq \pa_t  \Phi(t,[m])+\frac{1}{2} \sup_{i, a\in \Pc(\{0,1\}^K)} Tr\left(D_{x}D_m \Phi\left(t,m,[m]\right)\left(\sum_{j}{a(j)}\left( \1_{i\in j}e_{j^c}e_{j^c}^\top+\1_{i\notin j}e_{j}e_{j}^\top \right)\right)\right) \\
& \geq \pa_t \Phi(t,[m])+ \frac{1}{2} \sup_{a \in \Pc(\{0,1\}^K)} \sum_i \\
& \quad \quad  \Phi_i(t, m) \, Tr\left(D_{x}D_m \Phi\left(t,m,[m]\right)\left(\sum_{j}{a(j)}\left( \1_{i\in j}e_{j^c}e_{j^c}^\top+\1_{i\notin j}e_{j}e_{j}^\top \right)\right)\right),
\end{align*}
where $\Phi_i(t,m)$ denotes the $i$-th coordinate of $D_m \Phi(t,m,[m])$. This proves the supersolution property we want. 
\end{proof}

\begin{rmk}
It can be easily verified that smooth supersolutions of 
\begin{align*}
& 0=\pa_t   \phi(t,x)+\frac{1}{2}\sup_{a\in \Pc(\{0,1\}^K)} \sum_i \\
& \quad \quad \pa_{x^i} \phi(t,x) \, Tr\left(D^2_{xx}\phi\left(t,x\right)\left(\sum_{j}{a(j)}\left( \1_{i\in j}e_{j^c}e_{j^c}^\top+\1_{i\notin j}e_{j}e_{j}^\top \right)\right)\right)
\end{align*}
cannot generate supersolutions of \eqref{eq:originalPDE} simply by integrating $x$ over $m$. 
\end{rmk}

We now show how we can use the Lemma \ref{eq:simplepde} to obtain regret bounds. 
Fix a large time horizon $T$. Denote $\widetilde{m}:=m^{* \frac{1}{\sqrt{T}}}$, $t_n=\frac{n}{T}$, where $n$ denotes the current step. For any smooth supersolotuion $\phi$ of \eqref{eq:simplepde}, we define a strategy of the forecaster
\begin{align*}
(\b_0^*,\dotso, \b_{T-1}^*)
\end{align*}
via
\begin{align}\label{eq:forecasterstrategy}
\b_n^*(m):=D_m \Phi \left(t_{n},\tilde m, \left[\widetilde{m} \right]\right). 
\end{align}
Suppose that the initial belief is $m_0$, and denote random belief as $(m_n)_{n=1,\dotso,T}$. Then it is clear that 
\begin{align*}
v_T(m_0) \leq \sup_{\alpha} \E^{\b^*, \alpha} \left[f([m_T]) \right],
\end{align*}
where $f$ is the terminal condition $f(x):= \max_i x^i$. The following Proposition provides assumptions for such a methodology to yield to regret bounds. 

\begin{proposition}\label{prop:upperbound}
Suppose $\phi$ is a classical solution of \eqref{eq:simplepde} and
\begin{align}\label{eq:upperderivative}
|\pa^2_{tt} \phi(t,x)| \leq \frac{C}{(1-t)^{3/2}}, \quad |\pa^3_{xxx} \phi(t,x)|+ |\pa^2_{tx} \phi(t,x)|  \leq \frac{C}{1-t}, \quad \forall x \in \R^K,
\end{align}
for some positive constant $C$. Then the strategy $\b^*$ of the forecaster defined in \eqref{eq:forecasterstrategy} yields regret bounded above by $\sqrt{T}  \phi(0, [ \widetilde m_0])$ asymptotically. 
\end{proposition}

\begin{proof} 
Our goal is to show that $\lim\limits_{T \to \infty} \frac{1}{\sqrt{T}}\sup_{\alpha} \E^{\b^*, \alpha} \left[f([m_T]) \right]- {\phi}\left( 0,\left[ \widetilde{ m}_0 \right] \right)  \leq 0$.  First we rewrite the difference as a telescopic sum
\begin{align*}
& \frac{1}{\sqrt{T}}\sup_{\alpha} \E^{\b^*, \alpha} \left[f([m_T]) \right]- {\phi}\left( 0,\left[ \widetilde{ m}_0 \right] \right) = \sup_{\alpha} \E^{\b^*, \alpha} \left[f([ \widetilde{m}_T]) \right]- {\phi}\left( 0,\left[ \widetilde{ m}_0 \right] \right) \\
&= \sup_{\alpha} \sum_{n=0}^{T-1}  \left( \E^{\b^*,\alpha}\left[ {\phi}(t_{n+1}, [\widetilde{m}_{n+1} ]) \right] -  \E^{\b^*,\alpha}\left[{\phi}(t_{n}, [\widetilde{m}_{n} ]) \right] \right). 
\end{align*}
Conditioning on $\widetilde{m}_{n}=m$, we have that 
\begin{align}\label{eq:last1}
& \E^{\b^*,a}\left[ {\phi}(t_{n+1}, [\widetilde{m}_{n+1} ])- \phi(t_{n}, [\widetilde{m}_{n}])  \, | \, \widetilde{m}_{n}=m \right] \\
&= \sum_{i} \b^*_{n}(m_n)(i)  \left(\hat a(i) \phi \left(t_{n+1}, \left[A^{a, m}_{i,\sqrt{T}} \right]\right)+  \hat a(-i) \phi \left(t_{n+1}, \left[A^{a,  m}_{-i,\sqrt{T}} \right]\right) \right) - \phi(t_{n}, [m]). \notag
\end{align}
Using the linear structure of $\phi(t,[m])$, it can be seen that 
\begin{align}
&\phi \left(t_{n+1}, \left[A^{a,m}_{i,\sqrt{T}} \right]\right)-\phi(t_n, [m]) \notag \\
&= \sum_{j: i\in j} \frac{a(j)}{\hat a(i)} \int \left( \phi\left(t_{n+1}, x-\frac{e_{j^c}}{\sqrt T} \right)-\phi(t_n,x)  \right) m(dx). \label{eq:last2}
\end{align}
For any $i \in j \subset [K]$, we have the equality
\begin{align}\label{eq:f1}
&\phi\left(t_{n+1}, x-\frac{e_{j^c}}{\sqrt T} \right)-\phi(t_n,x)  \\ 
&= -\frac{e_{j^c}}{\sqrt T}  \pa_x \phi(t_n,x) \label{eq:f2}+\frac{1}{T} \left(\pa_t \phi(t_n,x)+\frac{1}{2} e_{j^c}^\top \pa^2_{xx} \phi(t_n,x) e_{j^c}  \right) \notag  \\
& \quad   + \frac{1}{T} \int_0^1 \pa_t \phi\left(t_n+\frac{s}{T},x-\frac{se_{j^c}}{\sqrt T}\right)-\pa_t \phi(t_n,x) \, ds \notag  \\
& \quad - \frac{e_{j^c}}{\sqrt T} \int_0^1   \pa_x \phi \left(t_n+\frac{s}{T},x- \frac{se_{j^c}}{\sqrt T} \right)-\pa_x \phi \left(t_n,x- \frac{se_{j^c}}{\sqrt T} \right) ds  \notag \\
& \quad +\frac{1}{T} \int_0^1 (1-s) e_{j^c}^\top \left(\pa_{xx}^2 \phi\left(t_n,x-\frac{se_{j^c}}{\sqrt T} \right)-\pa_{xx}^2 \phi(t_n,x) \right) e_{j^c} \, ds. \notag
\end{align}

Using our assumption \eqref{eq:upperderivative}, we can estimate the last three terms in the equation above
\begin{align*}
&\left| \frac{1}{T}\int_0^1 \pa_t\phi \left(t_n+\frac{s}{T},x- \frac{se_{j^c}}{\sqrt T} \right) - \pa_t\phi \left(t_n,x\right) ds \right| \leq  C \int_0^{1/T} \frac{\frac{1}{T}-s}{(1-t_n-s)^{3/2} } ds  \\
& \left| \frac{e_{j^c}}{\sqrt T} \int_0^1   \pa_x \phi \left(t_n+\frac{s}{T},x- \frac{se_{j^c}}{\sqrt T} \right)-\pa_x \phi \left(t_n,x- \frac{se_{j^c}}{\sqrt T} \right) ds \right|  \leq C \sqrt T \int_0^{1/T} \frac{\frac{1}{T}-s}{1-t_n-s} ds \\
& \left|\frac{1}{T} \int_0^1 (1-s) e_{j^c}^\top \left(\pa_{xx}^2 \phi\left(t_n,x-\frac{se_{j^c}}{\sqrt T} \right)-\pa_{xx}^2 \phi(t_n,x) \right) e_{j^c} \, ds  \right| \leq \frac{C}{T^{3/2}(1-t_n)}.
\end{align*}
Let us define 
\begin{equation}\label{eq:errorterm}
O(T,n):=C \left( \int_0^{1/T} \frac{\frac{1}{T}-s}{(1-t_n-s)^{3/2} } ds+ \sqrt T \int_0^{1/T} \frac{\frac{1}{T}-s}{1-t_n-s} ds +\frac{1}{T^{3/2}(1-t_n)} \right). 
\end{equation}
Now plugging \eqref{eq:last2} and \eqref{eq:f1} into \eqref{eq:last1}, we obtain that 
\begin{equation}\label{eq:ff}
\begin{split}
& \E^{\b^*,a}\left[ {\phi}(t_{n+1}, [\widetilde{m}_{n+1} ])- \phi(t_{n}, [\widetilde{m}_{n}])  \, | \, \widetilde{m}_{n}=m \right] \\
& \leq \frac{1}{\sqrt T}\sum_i \b^*_n( m_n)(i) \left( \sum_{j, i \not \in j} a(j) e_j^\top-\sum_{j:i \in j} a(j)e_{j^c}^\top   \right) \pa_x \phi(t_n,[m])+\frac{1}{T} \times  \\
& \left(\pa_t \phi(t_n, [m])+\frac{1}{2}\sum_{i} \b_n^*(m_n)(i) Tr\left(D^2_{xx}\phi\left(t_n,[m]\right)\left(\sum_{j}{a(j)}\left( \1_{i\in j}e_{j^c}e_{j^c}^\top+\1_{i\notin j}e_{j}e_{j}^\top \right)\right)\right) \right)  \\
&+O(T,n).
\end{split}
\end{equation}
The first term on the right hand side vanishes due to our choice of $\b^*$, the second term is non-positive due to the supersolution property of $\phi$, and thus we obtain that 
\begin{align*}
\E^{\b^*,a}\left[ {\phi}(t_{n+1}, [\widetilde{m}_{n+1} ])- \phi(t_{n}, [\widetilde{m}_{n}])  \, | \, \widetilde{m}_{n}=m \right] \leq O(T,n). 
\end{align*}
Summing up from $n=0$ to $T-1$, taking supremum over $\alpha \in \mathcal{A}$, and letting $T \to \infty$, we conclude that 
\begin{align*}
\lim\limits_{T \to \infty} \frac{1}{\sqrt{T}}\sup_{\alpha} \E^{\b^*, \alpha} \left[f([m_T]) \right]- {\phi}\left( 0,\left[ \widetilde{ m}_0 \right] \right)  \leq \lim\limits_{T \to \infty} \sum_{n=0}^{T-1} O(T,n)= 0.
\end{align*}
\end{proof}

\begin{example}\label{ex:1}

Let us take $\phi$ to be the smooth solution of the following heat equation
\begin{align*}
\begin{cases}
\pa_t \phi+ \frac{1}{2} \Delta \phi=0 & \text{on $\mathbb{R}^K \times [0,1)$ }\\
\phi(1, x)= f(x) & \text{on $\mathbb{R}^K \times \{1\}$ }. 
\end{cases}
\end{align*}
It can be easily verified as in \cite[Proposition 19]{10.5555/3546258.3546330} that $\phi$ satisfied \eqref{eq:upperderivative}. According to \cite[Appendix F.1]{pmlr-v125-kobzar20a}, we know that $D^2_{x^lx^k} \phi(t,x) >0$ if $l=k$ and $D^2_{x^lx^k} \phi(t,x) <0$ if $l \not =k$. Therefore for any $i \in [K]$ and $j \subset [K]$, we have that 
\begin{align*}
\frac{1}{2}Tr \left(D^2_{xx} \phi(t,x) \left(  \1_{i\in j}e_{j^c}e_{j^c}^\top+\1_{i\notin j}e_{j}e_{j}^\top\right)  \right) \leq \frac{1}{2} \Delta \phi(t,x), 
\end{align*} 
and hence 
\begin{align*}
\frac{1}{2}\sup_{i, a\in \Pc(\{0,1\}^K)} Tr\left(D^2_{xx} \phi\left(t,x\right)\left(\sum_{j}{a(j)}\left( \1_{i\in j}e_{j^c}e_{j^c}^\top+\1_{i\notin j}e_{j}e_{j}^\top \right)\right)\right) \leq \frac{1}{2} \Delta \phi(t,x).
\end{align*}
Thus $\phi$ is a smooth supersolution of \eqref{eq:simplepde} which satisfies \eqref{eq:upperderivative} according to \cite{pmlr-v125-kobzar20a}. By Feynman-Kac formula, we have $\phi(0,x)=\E_x[f(N^1,N^2,\dotso,N^K)]$ where $N^i$ is a standard normal. Supposing $x=(0,\dotso,0)$, then by Jensen's inequality  we have that for any $t \geq 0$
\begin{align*}
e^{t \E[f(N^1,\dotso,N^K)]} \leq \E[e^{tf(N^1,\dotso,N^K)} ] \leq K \E[e^{tN^1}]=Ke^{t^2/2},
\end{align*}
and hence $\E[f(N^1,\dotso,N^K)] \leq \frac{\log K}{t} + \frac{t}{2}$. Choosing $t =\sqrt{2 \log K}$, we obtain that $\phi(0,0) \leq \sqrt {2 \log K}$. Therefore, when initial belief is $\delta_{0}$, in our game where both agents have partial information, the asymptotic regret is bounded above by $\sqrt{2T\log K }$. It is smaller than the expected regret $5.15\sqrt{TK \log K }+\sqrt{\frac{TK}{\log K}}$in the case of adversarial bandit where both agents only observe $Y_t$ \cite[Theorem 3.4]{BubeckCesaBianchi2012}. The regret bound we obtain is two times larger than the performance of multiplicative weight algorithms obtained in \cite{gravin2017tight}.

\end{example}

\begin{rmk}
Our main contribution in terms of regret bound is to extend the PDE based methodology of \cite{pmlr-v125-kobzar20a} to the version bandit problems we study. In Lemma \ref{lem:simplesuper}, this bound is obtained by considering a functional linear in $m$ in the sense that $\Phi(t,m)=\int \phi(t,x)m(dx)$. 
Similar to \cite{pmlr-v125-kobzar20a}, the PDE tools are expected to yield sharper bounds by considering more sophisticated supersolutions to \eqref{eq:simplepde}. 

For example, any solution of 
\begin{align}\label{eq:originalPDE2}
0&=\pa_t  u(t,m) +\sup_{i,a\in \Pc(\{0,1\}^K)}  \\
&+\frac{1}{2}\left(\Vc_{a,i}^\top D^2_{mm} u\left(t,m,[m],[m]\right)\Vc_{a,i}+\sum_{j:i\in j}\frac{a(j)}{\hat a(i)}e_{j^c}^\top D_{x}D_m u\left(t,m,[m]\right)e_{j^c}\right)\notag\\
&+\frac{1}{2} \left(\Vc_{a,-i}^\top D^2_{mm} u\left(t,m,[m],[m]\right)\Vc_{a,-i}+\sum_{j:i\notin j}\frac{a(j)}{\hat a(-i)}  e_{j}^\top D_{x}D_m u\left(t,m,[m]\right)e_{j}\right) \notag
\end{align}
is a supersolution of \eqref{eq:simplepde}. 
For all $i\in [K]$ and $ a\in \Pc(\{0,1\}^K)$, we can define the symmetric matrices
\begin{align*}
\Sigma(i,a)&=\left({a(e_i)}\Vc_{a,e_i}\Vc_{a,e_i}^\top+{a(-e_i)} \Vc_{a,-e_i}\Vc_{a,-e_i}^\top\right)\\
\tilde \Sigma(i,a)&=\left({a(e_i)}\sum_{j:i\in j}\frac{a(j)}{a(e_i)}e_{j^c}^\top e_{j^c}+{a(-e_i)} \sum_{j:i\notin j}\frac{a(j)}{a(-e_i)}  e_{j}^\top e_{j}\right)-\Sigma(i,a). \notag
\end{align*}
By computing $v^\top \Sigma(i,a)v$ and $v^\top \tilde \Sigma(i,a)v$ for $v\in \R^K$, one can show that these matrices are non-negative. Thus, \eqref{eq:originalPDE2} can be written as the Hamilton-Jacobi-Bellman equation
\begin{align}
0&=\pa_t  u(t,m)+\frac{1}{2}\sup_{i, a\in \Pc(\{0,1\}^K)}  Tr\left(\Hc  u\left(t,m\right)\Sigma\left(i,a\right)+D_{x}D_m u\left(t,m,[m]\right)\tilde \Sigma\left(i,a\right)\right)\label{eq:wpde3}
\end{align}
where in line with \cite{chow2019partial}, the term $$\Hc U(t,m):=\int D_xD_m U(t,m,x)m(dx)+\int\int D^2_{mm} U(t,m,x,y)m(dx)m(dy)$$ is the so-called the Wasserstien Hessian of $U(t,\cdot)$.
A simple computation shows that the value function corresponding to a controlled version of \cite[Equation (1.8)]{chow2019partial} would yield to a viscosity solution to \eqref{eq:wpde3}; see \cite[Remark 3.5]{chow2019partial}. Then, this value function can be used as a supersolution of \eqref{eq:simplepde} (which would indeed depend nonlinearly on $m$). However such a methodology requires a comparison result for viscosity solutions of \eqref{eq:wpde3} (or smoothness of the value function) to obtain regret bounds. This comparison result and computation of improved regret bounds via a nonlinear $\Phi$ is being addressed by the authors on an ongoing work.

\end{rmk}
\section{Lower bound by smooth subsolution of the PDE}
As in the last section, we construct strategies for the adversary using smooth subsolutions of \eqref{eq:originalPDE}. Recall that $\mathcal{E}$ is the set of balanced strategies defined in Remark~\ref{rmk:balanced}. The proof of following lemma is almost the same as Lemma~\ref{lem:simplesuper} and thus we omit it. 

\begin{lemma}\label{lem:simplesub}
Let $ \phi$ be a smooth solution of
\begin{equation*}
\begin{split}
&0 \leq \pa_t   \phi(t,x)+\frac{1}{2}\inf_{i} \, Tr\left(D^2_{xx}\phi\left(t,x\right)\left(\sum_{j}{a_t(j)}\left( \1_{i\in j}e_{j^c}e_{j^c}^\top+\1_{i\notin j}e_{j}e_{j}^\top \right)\right)\right)  \\
&\phi(1,x) \leq \max_i x^i, \, \, \phi(t,x + \lambda \1)=\phi(t,x)+\lambda,
\end{split}
\end{equation*}
where $a_t \in \mathcal{E}, \, t \in [0,1], m \in \Pc(\R^K)$ are balanced strategies. Then, the function $\Phi:[0,1]\times \Pc_2(\R^K)\mapsto \R$ defined by 
$$ \Phi(t,m)=\phi(t,[m])=\int \phi(t,x) \, m(dx)$$
is a smooth subsolution to \eqref{eq:originalPDE}.
\end{lemma}

\begin{rmk}
Note that in Lemma~\ref{lem:simplesub}, the choice of balanced strategies $a_t$ only depends on time $t$. 
\end{rmk}

Given balanced strategies $(a_t)_{t \in [0,1]}$ and subsolution $\phi$ as in Lemma~\ref{lem:simplesub}, we construct strategies for the adversary in the original game \eqref{eq:minmax}. For a large time horizon $T$. Let us denote $t_n=\frac{n}{T}$, where $n$ is the current step. We define a strategy $\a^*$ of the adversary via
\begin{align*}
\a^*_n=a_{t_n}, \quad n=0,\dotso, T-1. 
\end{align*}

\begin{proposition}\label{prop:simplelowerbound}
Suppose $(a_t)_{t \in [0,1]}$, $\phi$ are balanced strategies and classical solutions as in Lemma~\ref{lem:simplesub} that satisfies \eqref{eq:upperderivative}.Let $m_0$ be the initial belief. Then the strategy $\a^*$ of the adversary defined yields regret bounded below by $\sqrt{T}  \phi\left(0, \left[  m_0^{* \frac{1}{\sqrt T}}\right]\right)$ asymptotically. 
\end{proposition}
\begin{proof}
The argument is almost the same as that of Proposition~\ref{prop:upperbound}. Just notice that \eqref{eq:ff} now becomes
\begin{equation*}
\begin{split}
& \E^{b,\a^*}\left[ {\phi}(t_{n+1}, [\widetilde{m}_{n+1} ])- \phi(t_{n}, [\widetilde{m}_{n}])  \, | \, \widetilde{m}_{n}=m \right] \\
& \geq \frac{1}{\sqrt T}\sum_i b(i) \left( \sum_{j, i \not \in j} \a^*_n(j) e_j^\top-\sum_{j:i \in j} \a^*_n(j)e_{j^c}^\top   \right) \pa_x \phi(t_n,[m])+\frac{1}{T} \times  \\
& \left(\pa_t \phi(t_n, [m])+\frac{1}{2}\sum_{i} b(i) Tr\left(D^2_{xx}\phi\left(t_n,[m]\right)\left(\sum_{j}{\a_n^*(j)}\left( \1_{i\in j}e_{j^c}e_{j^c}^\top+\1_{i\notin j}e_{j}e_{j}^\top \right)\right)\right) \right)  \\
&+O(T,n),
\end{split}
\end{equation*}
where $O(T,n)$ is defined in \eqref{eq:errorterm}. The first order term on the right hand vanishes since $\a^*_n$ is balanced, and second order term is nonnegative due to the subsolution property of $\phi$. Thus we obtain that $\E^{b,\a^*}\left[ {\phi}(t_{n+1}, [\widetilde{m}_{n+1} ])- \phi(t_{n}, [\widetilde{m}_{n}])  \, | \, \widetilde{m}_{n}=m \right] \geq O(T,n)$. Then summing up from $n=0$ to $T-1$, taking infimum over $\b \in \mathcal{B}$, and letting $T \to \infty$, we conclude our result. 

\end{proof}

\begin{example}
Let us take $a_t$ to be the uniformly distribution over $\{0,1\}^K$ for each $t \in [0,1]$. Then it can be easily verified that 
\begin{align*}
\sum_{j}{a_t(j)}\left( \1_{i\in j}e_{j^c}e_{j^c}^\top+\1_{i\notin j}e_{j}e_{j}^\top \right) = \frac{1}{4} e_{[K]} e_{[K]}^\top  + \frac{1}{4} I_K, \quad \forall \, i \in [K].
\end{align*} 
where $I_K$ stands for the identity matrix of dimension $K \times K$.

Let us take $\phi$ to be the smooth solution of the following heat equation
\begin{align*}
\begin{cases}
\pa_t \phi+ \frac{1}{8} \Delta \phi=0 & \text{on $\mathbb{R}^K \times [0,1)$ }\\
\phi(1, x)= f(x) & \text{on $\mathbb{R}^K \times \{1\}$ }. 
\end{cases}
\end{align*}
It can be easily seen that such $\phi$ and $(a_t)_{t \in [0,1]}$ satisfy all the assumptions in Proposition~\ref{prop:simplelowerbound}. Therefore when initial belief is $\delta_{0}$, the asymptotic asymptotic regret is bounded below by $\sqrt{T} \phi(0,0)$. By Feynman-Kac formula, $\phi(0,0)=\E[f(N^1,\dotso, N^K)]$ where $N^i$ is gaussian distributed with mean $0$ and variance $1/4$ for each $i=1,\dotso, K$. Then according to \cite[Theorem 3]{2015arXiv151102176O}, we obtain a lower bound $\phi(0,0) \geq 0.065 \sqrt {\log K}-0.35$. 

\end{example}

\bibliographystyle{siam}
\bibliography{ref.bib}

\end{document}